\documentclass[psamsfonts]{amsart}
\usepackage[utf8]{inputenc}
\usepackage{amsfonts}
\usepackage{hyperref}
\hypersetup{
pdftitle={On the quantum differentiation of smooth real-valued functions},
pdfsubject={Mathematics, Quantum Algebra, Q-calculus, Time-scale calculus},
pdfauthor={Kolosov Petro},
pdfkeywords={Quantum calculus, Quantum algebra, Power quantum calculus, Quantum difference, q-derivative, Jackson derivative, q-calculus, q-difference, Time Scale Calculus, Series Expansion, Taylor's theorem, Taylor's formula, Taylor's series, Taylor's polynomial, Analytic function, Series representation, Derivative, Differential calculus, Difference Equations, Numerical Differentiation, Polynomial, Exponential function, Exponentiation, Binomial coefficient, Binomial theorem, Binomial expansion, Mathematics, Numerical analysis, Mathematical analysis, arXiv, Preprint}
}
\usepackage{amsmath}
\usepackage{xcolor}
\usepackage{amsthm}
\usepackage{pdflscape}
\usepackage{pgfplots}
\usepackage{mathrsfs}
\usepackage[backend=bibtex]{biblatex}  %
\newtheorem{thm}{Theorem}[section]

\newtheorem{lem}[thm]{Lemma}

\theoremstyle{definition}

\theoremstyle{remark}

\makeatletter
\let\c@equation\c@thm
\makeatother
\numberwithin{equation}{section}

\title{On a Special Case of Dirichlet's Theorem}

\author{Jhixon Macías}
\address{Jhixon Macías\newline
University of Puerto Rico at Mayaguez, Mayaguez, PR, USA\newline
United States of America}
\email{jhixon.macias@upr.edu}



\begin{document}

\begin{abstract}
Let $p$ be a prime number, and $h$ a positive integer such that $\gcd(p,h)=1$. We prove, without invoking Dirichlet's theorem, that the arithmetic progression $p\left(\mathbf{N}\cup \{0\}\right)+h$ contains infinitely many prime numbers. This is a special case of Dirichlet's theorem not considered by other authors.
\end{abstract}

\maketitle
\section{Introduction}
Let $\mathbf{N}$ denote the set of positive integers. Dirichlet's Theorem \cite{dirichlet1837beweis} asserts that for two integers $h$ and $k>0$ such that $\gcd(h,k)=1$, the arithmetic progression
\begin{equation*}\label{e1}
    \mathbf{A}(h,k):=k\left(\mathbf{N}\cup \{0\}\right)+h
\end{equation*}
contains infinitely many prime numbers. Dirichlet's Theorem is established as a consequence of the following related asymptotic formula:
\begin{equation*}\label{e2}
    \sum_{\substack{p\leq x\\ p\equiv h\pmod k}}\frac{\log p}{p}=\frac{1}{\varphi(k)}\log x + O(1),
\end{equation*}
where $x>1$, and the sum extends over primes $p$ that are less than or equal to $x$ and congruent to $h$ modulo $k$.

Proving Dirichlet's Theorem is a challenging task; a proof can be found in \cite[chapter 7]{apostol1998introduction}. However, several authors later simplified Dirichlet's original proof. For instance, in 1949, Selberg \cite{selberg1949elementary} provided a non-analytic proof of this result, and in 1950, Harold N. Shapiro \cite{shapiro1950primes} also published a proof without using Dirichlet series. Despite these efforts, obtaining a straightforward non-analytic proof of this result remains elusive, except for special cases. Among these special cases are those of $\mathbf{A}(h,k)$ for certain tuples of positive integers $(h,k)$, such as $(1,4), (3,4), (4,5)$, $(3,8)$, $(5,8)$, and $(7,8)$, as detailed in \cite{apostol1998introduction, sierpinski1988elementary}.

Furthermore, there are additional special cases of Dirichlet's Theorem that are more general than those previously mentioned. For instance, according to \cite[Theorem 11]{erdHos2003topics}, for every positive prime $p$ and positive integer $c$, there exist infinitely many primes in $\mathbf{A}(1,p^c)$. Additionally, as proven in \cite[Theorem 1]{gauchman2001special}, for every positive integer $k$, there are infinitely many primes in $\mathbf{A}(1,k)$.

In this article, we establish the existence of infinitely many prime numbers for a special case of Dirichlet's Theorem that has not been explored by other authors. We prove that for any prime number $p$ and any positive integer $h$ such that $\gcd(h,p)=1$, the arithmetic progression $\mathbf{A}(h,p)$ contains infinitely many prime numbers. 

\section{Main Theorem}
Let $p_n$ denote the $n$-th prime. Then, $s:= \liminf (p_{n+1}-p_{n})<\infty$, i.e., there are infinitely many pairs of primes that differ by $s<\infty$, as shown in \cite{zhang2014bounded}.
Let us begin with the following Lemma,  which is proven without invoking Dirichlet's Theorem.
\begin{lem}\label{l1}
Let $k=\prod_{i=1}^{r}p_i>s$ a square-free number where each $p_i$ is a prime number greater than $2$. Then, $A(h,k)$ contains infinitely many prime numbers for every positive integer $h$ such that $\gcd(h,k)=1$.
\end{lem}
\begin{proof}
Note that it is sufficient to verify that the proposition holds for every positive integer $h<k$ such that $\gcd(h, k) = 1$. Now, for every $n\in\mathbf{N}$, we define
\begin{equation*}
    \mathbf{G}(n):=\{m\in\mathbf{N} \ \ | \ \ \gcd(n,m)=1\}.
\end{equation*}
Let $p$ be a prime number. Then, it is clear that  $\gcd(m, p)>1$ if and only if $m=kp$ for some positive integer $k$. Consequently, we have that $\displaystyle \mathbf{G}(p)=\mathbf{N}\setminus \mathbf{M}(p),$
where $\mathbf{M}(p):=\{p,2p,3p, \dots \}$ is the set of multiples of $p$. It is straightforward to prove that, since $p$ is prime, the following equality holds:
\begin{equation*}
    \mathbf{N}\setminus\mathbf{M}(p)=\bigcup_{h=1}^{p-1}\mathbf{A}(h,p), \ \ \ \text{and therefore,} \ \ \  \mathbf{G}(p)= \bigcup_{h=1}^{p-1}\mathbf{A}(h,p).
\end{equation*}
So, for $k$ we have that:
\begin{equation*}
    \mathbf{G}(k)=\bigcap_{i=1}^r\mathbf{G}(p_i)=\bigcap_{i=1}^r\bigcup_{h=1}^{p_i-1}\mathbf{A}(h,p_i)=\bigcup_{\substack{1\leq h< k\\ \gcd(h,k)=1}}\mathbf{A}(h,k)
\end{equation*}
Note that each $\mathbf{A}(h,k)$ is disjoint from the others. Obviously, $\mathbf{G}(k)$ contains infinitely many primes, in particular, all prime numbers greater than $k$. Thus, since $\mathbf{G}(k)$ is a finite union of disjoint arithmetic progressions, at least one of them contains infinitely many primes. In other words, there exists a positive integer $1\leq h_1<k$ with $\gcd(h_1,k)=1$ such that $\mathbf{A}(h_1,k)$ contains infinitely many primes.

Now, let us consider the set
\begin{equation*}
    \mathbf{G}(h_1,k):=\mathbf{G}(k)\setminus\mathbf{A}(h_1,k)=\bigcup_{\substack{1\leq h< k\\ \gcd(h,k)=1\\
    h\neq h_1}}\mathbf{A}(h,k).
\end{equation*}
Suppose that $\mathbf{G}(h_1,k)$ contains only a finite number of primes. Then, there exists a positive integer $N_0$ such that for $n\geq N_0$, we have that $p_n=k\cdot f(n)+h_1,$
where $f(n)$ is an integer depending on $n$. Note that $f(n)$ is increasing for $n\geq N_0$. Indeed, in this case, we have that $k\cdot[f(n+1)-f(n)]=p_{n+1}-p_n>0,$
which implies $f(n+1)-f(n)>0$. In particular, since $p_{n+1}-p_{n}$ cannot be equal to $k$, we have that $f(n+1)-f(n)>1$. Then, for $n\geq N_0$, we obtain $p_{n+1}-p_n=k\cdot [f(n+1)-f(n)]>p> s,$ which implies that there can only be a finite number of consecutive primes that differ by less than $s$. This is absurd. Hence, $\mathbf{G}(h_1,k)$, which is a finite union of disjoint arithmetic progressions, contains infinitely many primes. Consequently, there exists a positive integer $h_2\neq h_1$ with $\gcd(h_2,k)=1$ and $h_2<k$ such that $\mathbf{A}(h_2,k)$ contains infinitely many primes.

Consider the set:
\begin{equation*}
    \mathbf{G}(h_1,h_2,k):=\left(\mathbf{G}(k)\setminus\mathbf{A}(h_1,k)\right)\setminus\mathbf{A}(h_2,k)= \bigcup_{\substack{1\leq h< k\\ \gcd(h,k)=1\\
    h\neq h_1\\ h\neq h_2}}\mathbf{A}(h,k).
\end{equation*}
Suppose that $\mathbf{G}(h_1,h_2,k)$ contains only a finite number of primes. Then, there exists a positive integer $N_1$ such that for $n\geq N_1$, we have that $p_n=k\cdot f(n)+h_1 \ \ \text{or} \ \ p_n=k\cdot g(n)+h_2,$
where $f(n)$ and $g(n)$ are integers depending on $n$. Without loss of generality, consider the following extreme case for a fixed $n\geq N_1$: $p_{n+1}=k\cdot g(n+1)+h_2$ and $p_n=k\cdot f(n)+h_1$. We aim to obtain a lower bound for $g(n+1)-f(n)$, and the desired bound is $2$.

By the prime number Theorem, we have that
$$\displaystyle \lim_{n\to\infty}\frac{p_n}{n\log n}=1,$$
and using the inequality $\displaystyle \frac{k\cdot f(n)}{n\log n}<\frac{k\cdot f(n)+h_1}{n\log n}$, and the fact that $\displaystyle\lim_{n\to\infty}\frac{h_1}{n\log n}=0$, we obtain:
\begin{equation*}
\lim_{n\to\infty}\frac{k\cdot f(n)}{n\log n}=1, \ \ \text{similarly,} \ \ \lim_{n\to\infty}\frac{k\cdot g(n)}{n\log n}=1.
\end{equation*}
Consequently,
\begin{equation*}
\lim_{n\to\infty}\frac{f(n)}{g(n)}=\lim_{n\to\infty}\frac{g(n)}{f(n)}=1.
\end{equation*}
So, taking $0 < \epsilon \ll 1$, there exists an $N_2$ such that for $n\geq N_2$:
\begin{equation*}
    g(n+1)-f(n)>g(n+1)-\frac{g(n)}{1-\epsilon},
\end{equation*}
and since $g(n+1)\geq g(n)+2$ for every $n$, 
\begin{equation*}
    g(n+1)-f(n)\geq g(n)-\frac{g(n)}{1-\epsilon}+2>\footnote{Note that for $0 < \epsilon \ll 1$, $g(n)-\frac{g(n)}{1-\epsilon}>-1$.}1.
\end{equation*}
Hence, for $n\geq \max\{N_1,N_2\}$, we have that
\begin{equation*}
    p_{n+1}-p_n=k\cdot[g(n+1)-g(n)]+(h_2-h_1)>\footnote{Note that $|h_2-h_1|<k.$}2k-k=k> s.
\end{equation*}
This implies that there can only be a finite number of consecutive primes that differ by less than $s$. Again, this is absurd. Hence, $\mathbf{G}(h_1,h_2,k)$, which is a finite union of disjoint arithmetic progressions, contains infinitely many primes. Consequently, there exists a positive integer $h_3\notin \{h_1,h_2\}$ between $1$ and $k$ with $\gcd(h,k)=1$ such that $\mathbf{A}(h_3,k)$ contains infinitely many primes. Finally, repeat this process until all the $\mathbf{A}(h,k)$ are filled with prime numbers. 
\end{proof}
Now we are prepared to prove our main Theorem.
\begin{thm}\label{mth}
Let $p$ a prime number. Then, for every positive integer $h$ such that $\gcd(h,k)=1$, the arithmetic progression $\mathbf{A}(h,p)$ contains infinitely many prime numbers.
\end{thm}
\begin{proof}
For $p=2$, the proposition is obviously satisfied. For $p>s$, the proposition holds by Lemma \ref{l1}. Thus, we need to examine what happens when $2<p\leq s$. Let $p_1, p_2, \ldots, p_r$ be these prime numbers.

Consider $q>s$ as the first prime number greater than $s$ and take $k:=q\cdot\displaystyle\prod_{i=1}^{r}p_i$. Clearly, $k>s$. Thus, by Lemma \ref{l1}, for any positive integer $h$ such that $\gcd(h,k)=1$, the arithmetic progression $\mathbf{A}(h,k)$ contains infinitely many prime numbers . Then, note that $\displaystyle\mathbf{A}(h,k)=\mathbf{A}(h,q)\cap\left(\bigcap_{i=1}^r\mathbf{A}(h,p_i)\right)$ and therefore $A(h,k)\subset\mathbf{A}(h,p_i)$ for every $i=1,2,\dots , r$. Hence, $\displaystyle\mathbf{A}(h,p_i)$ contains infinitely prime numbers for every $i=1,2,\dots , r$.
\end{proof}

\section{Some Remarks}
We would like to mention the following remarks:
\begin{enumerate}
    \item The proposition that $s<\infty$ is entirely independent of Dirichlet's Theorem. Instead, as indicated by the author in \cite{zhang2014bounded}, it is a consequence of the work in \cite{goldston2009primes} combined with the ideas presented in \cite{Y5, B1, B2, B3}. This ensures that we avoid any circular reasoning with Dirichlet's Theorem, thus enabling us to proceed with our current research. Furthermore, a more elementary proof of this result can be examined, as proposed by James Maynard in \cite{maynard2015small}, whose proof is primarily based on the Selberg sieve.
    \item Dirichlet's Theorem is equivalent to the set of prime numbers being dense in the set of natural numbers with the \textit{Golomb Topology} (also called relatively prime integer topology) see \cite{golomb1959connected}. Since then, there has been interest in proving the density of the set of prime numbers in the natural numbers under the Golomb topology, with the goal of obtaining a topological proof of Dirichlet's Theorem. In this context, Theorem \ref{mth} implies  density of prime numbers set in the positive integers set with the \textit{Kirch topology} (also called prime integer topology) see \cite{steen1978counterexamples} for more details of this topology. While this does not constitute density of the prime numbers in the set of positive integers with the Golomb Topology per se, it motivates the search for a topological proof of Dirichlet's Theorem.
    \item Finally, we would like to mention that the proof of Lemma \ref{l1} and Theorem \ref{mth} is elementary in the sense that it does not a priori employ highly advanced techniques from the field, except perhaps for being aware that the proof of $s < \infty$ is not straightforward or elementary. The concept behind the proof of Lemma \ref{l1} is to consider $\mathbf{G}(k)$ as a (finite) set of "containers" being gradually filled one by one with prime numbers. Since there are infinitely many prime numbers available, the containers will, in principle, end up containing infinitely many prime numbers. In this analogy, the containers represent the arithmetic progressions $\mathbf{A}(h,k).$
\end{enumerate}
\printbibliography
\end{document}